\documentclass{amsart}
\usepackage{color}

\input xy
\xyoption{all}

\newtheorem{theorem}{Theorem}[section]
\newtheorem{lemma}[theorem]{Lemma}
\newtheorem{keylemma}[theorem]{Key Lemma}

\newtheorem{corollary}[theorem]{Corollary}
\newtheorem{proposition}[theorem]{Proposition}
\newtheorem{problem}[theorem]{Problem}
\newtheorem{claim}[theorem]{Claim}
\theoremstyle{definition}
\newtheorem{definition}[theorem]{Definition}
\newtheorem{remark}[theorem]{Remark}

\newcommand{\cs}{\mathrm{cs}}
\newcommand{\sbr}{\mathrm{sb}}
\newcommand{\w}{\omega}
\newcommand{\K}{\mathcal K}
\newcommand{\N}{\mathcal N}
\newcommand{\F}{\mathcal F}
\newcommand{\U}{\mathcal U}
\newcommand{\V}{\mathcal V}

\newcommand{\IN}{\mathbb N}
\newcommand{\suc}{\mathrm{succ}}
\newcommand{\A}{\mathcal A}
\newcommand{\Ra}{\Rightarrow}
\newcommand{\pr}{\mathrm{pr}}
\newcommand{\nw}{\mathrm{nw}}

\newcommand{\IR}{\mathbb{R}}

\newcommand{\sk}{sk}

\title{Sequential rectifiable spaces of countable $\cs^*$-character}
\author{Taras Banakh and Du\v san Repov\v s}
\address{T.Banakh: Ivan Franko National University of Lviv (Ukraine) and Jan Kochanowski University in Kielce (Poland)}
\email{t.o.banakh@gmail.com}
\address{D.Repov\v s: Faculty of Education, and
Faculty of Mathematics and Physics,
University of Ljubljana,
Kardeljeva Pl.~16,
Ljubljana, Slovenia 1000}
\email{dusan.repovs@guest.arnes.si}

\begin{document}
\begin{abstract} We prove that each non-metrizable sequential rectifiable space $X$ of countable $\cs^*$-character contains a clopen rectifiable submetrizable $k_\w$-subspace $H$ and admits a disjoint cover by open subsets homeomorphic to clopen subspaces of $H$. This implies that each sequential rectifiable space  of countable $\cs^*$-character is either metrizable or a topological sum of submetrizable $k_\w$-spaces. Consequently, $X$ is submetrizable and paracompact. This answers a question of Lin and Shen posed in 2011. 
\end{abstract}
\maketitle

\section{Introduction and Main Results}

In this paper we generalize to rectifiable spaces a result  of Banakh and Zdomskyy \cite{BZ} on the structure of sequential topological groups of countable $\cs^*$-character. They proved that any such group is either metrizable or contains an open submetrizable $k_\omega$-subgroup.

Rectifiable spaces were introduced by Arkhangel'ski\u\i\ as non-associative generalizations of topological groups. We recall that a topological space $X$ is {\em rectifiable} if there exist a point $e\in X$ and a homeomorphism $h:X\times X\to X\times X$ such that $h(x,e)=(x,x)$ and $h(\{x\}\times X)\subset \{x\}\times X$ for every $x\in X$. So, for every $x\in X$ the restriction $h|\{x\}\times X$ is a homeomorphism of $\{x\}\times X$ sending the point $(x,e)$ onto $(x,x)$. This means that the rectifiable space $X$ is topologically homogeneous.

Each topological group $G$ is a rectifiable space, this is witnessed by the homeomorphism $h:G\times G\to G\times G$, $h:(x,y)\mapsto(x,xy)$. On the other hand, there are examples of rectifiable spaces, which are not homeomorphic to topological groups. The simplest example of such space is the 7-dimensional sphere $S^7$ (see \cite[p.13]{Ar02} or \cite[\S3]{Us}).

Rectifiable spaces have many properties in common with topological groups. For example, rectifiable spaces are topologically homogeneous, rectifiable $T_0$-spaces are regular, first countable rectifiable $T_0$-spaces are metrizable, compact rectifiable spaces are Dugundji, etc. These and many other basic properties of rectifiable spaces were proved in \cite{Gul} and \cite{Us}. Some sequential and network properties of rectifiable spaces were discussed in \cite{Lin13}, \cite{LLL}, \cite{LS}.  The rectifiability of a topological space $X$ is equivalent to the existence of a compatible structure of topological lop on $X$. This equivalence will be discussed in Section~\ref{s:rectolop}. More information on rectifiable spaces can be found in the survey \cite[\S4]{Ar02}.

Having in mind that all rectifiable $T_0$-spaces are regular, we will restrict ourselves to regular spaces and will assume that all topological spaces in this paper are regular.

The aim of this paper is to recover the topological structure of sequential rectifiable topological spaces of countable $\cs^*$-character. Spaces of countable $\cs^*$-character were introduced and studied in \cite{BZ}. They are defined by means of $\cs^*$-networks. Next, we recall the definition of $\cs^*$-networks and some related notions.

Let $X$ be a topological space and $x\in X$. A family $\mathcal N$ of subsets of $X$ is called a {\em $\cs$-network} (resp. {\em $\cs^*$-network}) at $x\in X$ if for each neighborhood $O_x\subset X$ of $x$ and sequence $\{x_n\}_{n\in\w}\subset X$ that converges to $x$ in $X$ there exists a set $N\in\mathcal N$ such that $x\in N\subset O_x$ and $N$ contains all but finitely many (resp. infinitely many) elements of the sequence. It is clear that each $\cs$-network at $x$ is a $\cs^*$-network at $x$ while the converse is in general not true. However, if $\mathcal N$ is a countable $\cs^*$-network at $x$, then the family $\widetilde{\mathcal N}=\big\{\bigcup\F:\F\subset \N,\;|\F|<\w\big\}$ is a countable $\cs$-network at $x$ (see Lemma~\ref{l:cs=cs*}).

By the {\em $\cs$-character} $\cs_\chi(X,x)$ (resp. {\em $\cs^*$-character} $\cs^*(X,x)$) of a topological space $X$ at a point $x\in X$ we understand the smallest cardinality $|\mathcal N|$ of a $\cs$-network (resp. $\cs^*$-network) $\mathcal N$ at $x$. It is clear that $\cs^*_\chi(x,X)\le\cs_\chi(x,X)$. We know of no examples of  pointed topological spaces $(X,x)$ with $\cs_\chi^*(X,x)<\cs_\chi(X,x)$ (see Problem~1 in \cite{BZ}). For a topological space $X$ the cardinals $\cs_\chi(X)=\sup_{x\in X}\cs_\chi(X,x)$ and $\cs^*_\chi(X)=\sup_{x\in X}\cs^*_\chi(X,x)$ are called the {\em $\cs$-character} and the {\em $\cs^*$-character} of $X$, respectively.
It was observed in \cite{BZ} that a topological space $X$ has a countable $\cs^*$-character (at a point $x\in X$) if and only if $X$ has countable $\cs$-character (at $x$).

It is clear that each first-countable space $X$ has countable $\cs^*$-character. Another standard example of a space of countable $\cs^*$-character is any submetrizable $k_\omega$-space. A topological space $X$ is called
\begin{itemize}
\item {\em submetrizable} if $X$ admits a continuous metric;
\item a {\em $k_\omega$-space} (resp. {\em $sk_\w$-space}) if there exists a countable cover $\K$ of $X$ by compact (and metrizable) subspaces such that a set $A\subset X$ is closed in $X$ if and only if for any $K\in\K$ the intersection $A\cap K$ is closed in $K$;
\item a {\em $k$-space} if a subset $A\subset X$ is closed if and only if for any compact subset $K\subset X$ the intersection $K\cap A$ is closed in $K$;
\item {\em sequential} if for each non-closed subset $A\subset U$
there is a sequence $\{a_n\}_{n\in\w}\subset A$ that converges to a point $x\notin A$ in $X$;
\item {\em Fr\'echet-Urysohn} if for each subset $A\subset X$ and a point $a\in \bar A$ there is a sequence $\{a_n\}_{n\in\w}\subset A$, converging to the point $a$.
\end{itemize}
It can be shown that a topological space $X$ is an $sk_\w$-space if and only if it is a submetrizable $k_\w$-space.
It is clear that each Fr\'echet-Urysohn space is sequential and each sequential space is a $k$-space. Each $\sk_\w$-space is sequential (being a $k$-space with metrizable compact subsets).

To see that any $sk_\w$-space $X$ has countable $\cs^*$-character, fix a countable cover $\K$ of $X$ by compact metrizable subsets such that a subset $A\subset X$ is closed if and only if for every $K\in\K$ the intersection $A\cap K$ is closed in $K$.
We can assume that $\K=\{K_n\}_{n\in\w}$ for an increasing sequence $(K_n)_{n\in\w}$ of compact subsets of $X$. For every metrizable compact space $K_n$ fix a countable base $\mathcal B_n$ of its topology. Then the family $\bigcup_{n\in\w}\mathcal B_n$ is a countable $\cs$-network at each point $x\in X$, which implies that $X$ has countable $\cs^*$-character.

It turns out that in the class of sequential rectifiable spaces, metrizable spaces and $sk_\w$-spaces are the only sources of spaces of countable $\cs^*$-character. The following theorem is the main result of this paper. It follows immediately from Theorem~\ref{t:rectilop}, \ref{t:sum} and \ref{t:nonmetr}.

\begin{theorem}\label{main} Each non-metrizable sequential rectifiable space $X$ of countable $\cs^*$-character contains a clopen rectifiable $sk_\w$-subspace $H$ and admits a disjoint open cover by subspaces homeomorphic to clopen subspaces of $H$.
\end{theorem}

A subset of a topological space $X$ is called {\em clopen} if it is closed and open.
Theorem~\ref{main} implies the following corollary, which answers  Question 7.7 posed by Lin and Shen \cite{LS}.

\begin{corollary} Each sequential rectifiable space $X$ of countable $\cs^*$-character is paracompact and submetrizable.
\end{corollary}

 Theorem~\ref{main} was proved for topological groups by Banakh and Zdomskyy \cite{BZ} in a bit stronger form: {\em each non-metrizable sequential topological group $G$ of countable $\cs^*$-character is homeomorphic to the product $H\times D$ of an $sk_\w$-space $H\subset G$ and a discrete space $D$}. We do not know if this result remains valid for rectifiable spaces.

\begin{problem}\label{prob1new} Let $X$ be a non-metrizable sequential rectifiable space of countable $\cs^*$-character. Is $X$ homeomorphic to the product $H\times D$ of an $sk_\w$-space $H$ and a discrete space $D$?
\end{problem}

The answer to this problem is affirmative if the rectifiable space is locally connected.

\begin{corollary}\label{c:decomp} Any locally connected non-metrizable  sequential rectifiable space $X$ of countable $\cs^*$-character is homeomorphic to the product $H\times D$ of a connected $sk_\w$-space $H$ and a discrete space $D$.
\end{corollary}

\begin{proof} By Theorem~\ref{main}, the space $X$ contains a non-empty clopen $sk_\w$-subspace $U$. For any point $e\in H$ its connected component $H_e$ is contained in $U$ (by the connectedness of $H_e$) and is clopen in $X$ (by the local connectedness of $X$). Being a closed subset of the $sk_\w$-space $U$, the component $H_e$ is an $sk_\w$-space, too.

The topological homogeneity of $X$ implies that all connected component of $X$ are pairwise homeomorphic and open in $X$. Then the space $X$, being a topological sum of its connected components, is homeomorphic to the product $H_e\times D$ for some discrete space $D$ (whose cardinality is equal to the number of connected components of $X$).
\end{proof}

An affirmative answer to the following problem would imply an affirmative answer to Problem~\ref{prob1new}.

\begin{problem} Let $X$ be a non-metrizable rectifiable $sk_\w$-space. Is then every non-empty clopen subspace of $X$ homeomorphic to $X$?
\end{problem}

Next we briefly describe the structure of the paper. In Section~\ref{s:toptool} we develop some topological tools that are necessary for the proof of Theorem~\ref{main}. In Section~\ref{s:rectolop} we discuss the interplay between rectifiable spaces and topological lops (which are ``left'' generalizations of topological groups and loops). In particular, in Section~\ref{s:rectolop} we prove Theorem~\ref{t:sum} on decomposition of  locally cosmic topological lops into a direct topological sum of cosmic spaces. In Section~\ref{s:fans} we study properties of rectifiable spaces that contain closed copies of the metric fan $M_\w$ and the Fr\'echet-Urysohn fan $S_\w$, and prove that such rectifiable spaces cannot simultaneously be sequential and have countable $\cs^*$-character. For normal topological groups this result was proved by Banakh \cite{Ban98}, it was generalized to arbitrary topological groups by Banakh and Zdomskyy \cite{BZ} and to normal rectifiable spaces by F.~Lin, C.~Liu, and S.~Lin \cite{LLL}. Section~\ref{s:key} contains Key Lemma~\ref{l:key}, which is technically the most difficult result of the paper. In Section~\ref{s:metr}  Key Lemma~\ref{l:key} is applied to prove some metrizability criteria for rectifiable spaces of countable $\cs^*$-network and countably compact subsets in such spaces. In Section~\ref{s:nonmetr} we prove an algebraic version of Theorem~\ref{main} showing that each non-metrizable sequential topological lop of countable $\cs^*$-character contains a clopen $sk_\w$-sublop.

\section{Some topological tools}\label{s:toptool}

In this section we will discuss some additional topological tools which will be used in the proof of Theorem~\ref{main}.

One of such instruments is the notion of  a sequence tree. As usual, by a {\em tree} we understand a
partially ordered subset $(T,\le)$ such that for each $t\in T$ the
set $\downarrow t=\{\tau\in T:\tau\le t\}$ is well-ordered by
the order $\le$. Given an element $t\in T$ let $\uparrow
t=\{\tau\in T:\tau\ge t\}$ and let $\mathrm{succ}(t)=\min(\uparrow
t\setminus\{t\})$ be the set of successors of $t$ in $T$. A
maximal linearly ordered subset of a tree $T$ is called a {\em
branch} of $T$. We denote the set of maximal elements
of the tree $T$ by $\max T$.

\begin{definition} By a {\em sequence tree} in a topological space $X$ we understand a tree
$(T,\le)$ such that
\begin{itemize}
\item $T\subset X$;
\item $T$ has no infinite branches;
\item for each $t\notin\max T$ the set $\mathrm{succ}(t)$
of successors of $t$ is countable and converges to $t$.
\end{itemize}
\end{definition}

By saying that a subset $S$ of a topological space $X$ {\em converges} to a
point $x\in X$ we mean that for each neighborhood $O_x\subset X$ of
$x$, the set $S\setminus O_x$ is finite.
In the sequel, by a {\em convergent sequence} in a topological space $X$ we will understand a compact countable subset $S\subset X$ with a unique non-isolated point, which will be denoted by $\lim S$.

The following lemma is well-known \cite[Lemma~1]{BZ} and can be easily proven by
transfinite induction (on the sequential height $s(a,A)=\min\{\alpha:a\in
A^{(\alpha)}\}$ of a point $a$ in the closure $\bar A$ of a subset $A$ of a sequential space $X$).

\begin{lemma}\label{l1} A point $a\in X$ of a sequential topological space
$X$ belongs to the closure of a subset $A\subset X$ if and only if
there is a sequence tree $T\subset X$ such that $\min T=\{a\}$ and
$\max T\subset A$.
\end{lemma}

An important role in the theory of generalized metric spaces belongs to the following three test spaces: the metric fan $M_\w$, the Fr\'echet-Urysohn fan $S_\w$ and the Arens fan $S_2$ (see \cite{GTZ}, \cite{Lin}). The first two spaces will play an important role in our considerations, too.

The {\em metric fan} is the subspace
$$M_\w=\{(0,0)\}\cup\{(\tfrac1n,\tfrac1{nm}):n,m\in\IN\}\subset\IR^2$$ of the Euclidean plane. The metric fan is not locally compact at its unique non-isolated point $(0,0)$. By \cite[8.3]{vD}, a metrizable space is not locally compact if and only if it contains a closed subspace homeomorphic to the metric fan $M_\w$. Observe that for every $m\in\IN$ the compact set $I_m=\{(0,0\}\cup\{(\frac1n,\frac1{nm}):n\in\IN\}\subset M_\w$ converges to $(0,0)$.

By the {\em Fr\'echet-Urysohn fan} we understand the union $S_\w=\bigcup_{m\in\IN}I_m$ endowed with the strongest topology inducing the Euclidean topology on each convergent sequence  $I_m$, $m\in\IN$. The space $S_\w$ is a standard example of a Fr\'echet-Urysohn space which is not first-countable.

The {\em Arens fan} $S_2$ is the $k_\w$-space
$$S_2=\{(0,0)\}\cup\{(\tfrac1n,0):n\in\IN\}\cup\{(\tfrac1n,\tfrac1{nm}):n,m\in\IN\}$$
endowed with the strongest topology inducing the Euclidean topology on the convergent sequences
$$\{(0,0)\}\cup\big\{(\tfrac1n,0):n\in\IN\big\}\mbox{ \ and \ }\{(\tfrac1n,0)\}\cup\big\{(\tfrac1n,\tfrac1{nm}):m\in\IN\big\}\mbox{ \ for all $n\in\IN$}.
$$
The Arens fan $S_2$ is a standard example of a sequential space which is not Fr\'echet-Urysohn (see \cite[1.6.19]{En}).

We will also need some information on sequential barriers and $sb$-networks.
A subset $B$ of a topological space $X$ is called a {\em sequential barrier} at a point $x\in X$ if it contains all but finitely many points of every sequence that converges to $x$. It is clear that each neighborhood of $x$ is a sequential barrier at $x$. The converse is true if $X$ is  Fr\'echet-Urysohn at $x$. The latter means that each subset $A\subset X$ with $x\in\bar A$ contains a sequence that converges to $x$.

In general, sequential barriers need not be neighborhoods. For example, the (nowhere dense) set $B=\{(0,0)\}\cup\{(\frac1n,0):n\in\IN\}$ is a sequential barrier at the point $x=(0,0)$ of the Arens fan $S_2$.

A family $\mathcal B$ of sequential barriers at a point $x$ of a topological space $X$ is called an {\em $sb$-network} at $x$ if each neighborhood $O_x\subset X$ of $x$ contains some sequential barrier $B\in\mathcal B$.

A family $\mathcal N$ of subsets of a topological space $X$ is called a {\em network} if for each point $x\in X$ and neighborhood $O_x\subset X$ of $x$ there is a set $N\in\mathcal N$ such that $x\in N\subset O_x$. The smallest cardinality of a network is called the {\em network weight} of $X$ and is denoted by $\nw(X)$.
Regular spaces of countable network weight are called {\em cosmic} (see \cite[\S10]{Gru2}).

Finally, let us recall that a topological space $X$ is
\begin{itemize}
\item {\em sequentially compact} if each sequence in $X$ contains a convergent subsequence;
\item {\em countably compact} if each sequence in $X$ has an accumulation point in $X$.
\end{itemize}
It is clear that each sequentially compact space is countably compact.
A sequential space is sequentially compact if and only if it is countably compact.

For a convenience of the reader we give a proof of the following simple fact, first observed in \cite{BZ}.

\begin{lemma}\label{l:cs=cs*} A topological space $X$ has countable $\cs^*$-character \textup{(}at a point $x\in X$\textup{)} if and only if the space $X$ has countable $\cs$-character \textup{(}at $x$\textup{)}.
\end{lemma}

\begin{proof} The ``if'' part is trivial. To prove the ``only if'' part, fix a countable $\cs^*$-network $\mathcal N$ at $x$. Replacing $\mathcal N$ by a larger family, we can assume that $\mathcal N$ is closed under finite unions. We claim that in this case $\mathcal N$ is a $\cs$-network at $x$. Fix any neighborhood $O_x\subset X$ of $x$ and a sequence $(x_n)_{n\in\w}$ converging to $x$. Consider the countable subfamily $\mathcal M=\{N\in\mathcal N:N\subset O_x\}$ of $\mathcal N$ and let $\mathcal M=\{N_k\}_{k\in\w}$ be its enumeration. We claim that for some $n\in\w$ the set $M_n=\bigcup_{k\le n}N_k$ contains all but finitely many points of the sequence $(x_k)$, which is equivalent to saying that the set $\Omega_n=\{k\in\w:x_k\notin M_n\}$ is finite. It follows from $M_n\subset M_{n+1}$ that $\Omega_n\supset\Omega_{n+1}$ for all $n$. Assuming that every set $\Omega_n$, $n\in\w$, is infinite, we can find an infinite subset $\Omega\subset \w$ such that $\Omega\setminus \Omega_n$ is finite for all $n\in\w$. Since $\mathcal N$ is a $\cs^*$-network at the limit point $x$ of the sequence $(x_n)_{n\in\Omega}$,
there exists a set $N\in\mathcal N$ such that $N\subset O_x$ and $\{n\in\Omega:x_n\in N\}$ is infinite. It follows that $N\in\mathcal M$ and hence $N=N_k$ for some $k\in \omega$. Since $\Omega\setminus\Omega_k$ is finite, the set $\{n\in \Omega_k:x_n\in N=N_k\}$ is infinite. On the other hand, $\{n\in\Omega_k:x_n\in N_k\}\subset \{n\in\Omega_k:x_n\in M_k\}$ is empty by the definition of the set $\Omega_k$. This contradiction completes the proof.
\end{proof}

\section{Rectifiable spaces and topological lops}\label{s:rectolop}

In this section we discuss the relation of rectifiable spaces to topological-algebraic structures called topological lops. 
They are generalizations of topological loops \cite{HS} and are special kinds of topological magmas.

A {\em magma} is a set $X$ endowed with a binary operation $\cdot:X\times X\to X$. If this operation is associative then the magma is called a {\em semigroup}. For two elements $x,y\in X$ of a magma their product $\cdot(x,y)$ will be denoted by $x\cdot y$ or just $xy$.
An element $e\in X$ of a magma $X$ is called a {\em left unit} (resp. {\em right unit}\/) of $X$ if $ex=x$ (resp. $xe=x$) for all $x\in X$. An element $e\in X$ is a {\em unit} of $X$ if $xe=x=ex$ for all $x\in X$ (i.e., $e$ is both left and right unit in $X$).
It easy to see that any two units of a magma coincide, so a magma can have at most one unit.

For a point $a\in X$ of a magma $X$ the maps $L_a:X\to X$, $L_a:x\mapsto ax$, and $R_a:X\to X$, $R_a:x\mapsto xa$, are called  the {\em left shift} and the {\em right shift} of $X$ by the element $a$, respectively.

A magma $X$ is called
\begin{itemize}
\item a {\em loop} if $X$ has a unit $e$ and for every $a\in X$ the shifts $L_a:X\to X$ and $R_a:X\to X$ are bijective;
\item a {\em left-loop} if $X$ has a right unit $e$ and for every $a\in X$ the left shift $L_a:X\to X$ is bijective;
\item an {\em eleft-loop} if $X$ has a unit $e$ and for every $a\in X$ the left shift $L_a:X\to X$ is bijective.
\end{itemize}
It can be shown that a left-loop has a unique right unit $e$. In an eleft-loop this right unit is also a left unit. This justifies the prefix ``{\em eleft}'' -- $e$ is also a {\em left} unit. Since eleft-loops are central instruments in our subsequent considerations, and the term ``eleft-loop'' sounds a bit awkward, we will shorten it to ``lop'' (pretending that in ``loop'' the left ``o'' is responsible for left shifts whereas the right ``o'' for right shifts  --- so, ``lop'' is a ``loop'' with a removed right ``o'').

Therefore, a {\em lop} is a magma with unit and bijective left shifts. Equivalently, a lop can be defined as a magma with unit and bijective map $X\times X\to X\times X$, $(x,y)\mapsto (x,xy)$. On the other hand, a loop can be equivalently defined as a magma $X$ with unit and bijective maps $X\times X\to X\times X$, $(x,y)\mapsto (x,xy)$, and $X\times X\to X\times X$, $(x,y)\mapsto (x,yx)$.

These algebraic notions are related as follows:
$$\xymatrix{\mbox{associative left-loop}\ar@{<=>}[r]&\mbox{group}\ar@{=>}[r]&\mbox{loop}\ar@{=>}[r]&
\mbox{lop}\ar@{=>}[r]&\mbox{left-loop}.
}
$$
The first equivalence in this diagram is a standard exercise in  group theory (see, e.g. \cite[1.1.2]{Rob}).

For any points $a,b\in X$ of a lop $X$ the left shift $L_a:X\to X$ is a bijection of $X$, which implies that the set $a^{-1}b=\{x\in X:ax=b\}$ coincides with the singleton $\{L_a^{-1}(b)\}$. Sometimes it will be convenient to identify the singleton $a^{-1}b$ with the point $L_a^{-1}(b)$ thinking of $a^{-1}$ in the expression $a^{-1}b$ as the bijection $L_a^{-1}$ acting on the point $b$. Under this convention, for points $a,b,x$ of a lop the expression $a^{-1}b^{-1}x$ means $a^{-1}(b^{-1}(x))$ which is actually equal to $L_a^{-1}(L_b^{-1}(x))$. Analogously, we will understand longer expressions $a_1^{-1}a_2^{-1}\cdots a_n^{-1}x$ for points $a_1,\dots,a_n,x$ of a lop $X$. Observe that $x^{-1}x=e$ for  every $x\in X$.

The same convention concerns the multiplication of elements of a lop $X$: for elements $a_1,a_2,a_3\in X$ the expression $a_1a_2a_3$ means $a_1(a_2a_3)$, and so on. More precisely, for points $a_1,\dots,a_n\in X$ their product $a_1\cdots a_n$ is defined by induction as $a_1(a_2\cdots a_n)$ (since the multiplication of elements in a lop is not associative we should be careful with parentheses).

For two sets $A,B$ in a lop $X$ we put $AB=\{ab:a\in A,\;b\in B\}$ and $A^{-1}B=\{a^{-1}b:a\in A,\;b\in B\}$. More precisely, $A^{-1}B=\bigcup_{a\in A\;b\in B}a^{-1}b$.

A subset $A\subset X$ of a lop $X$ is called a {\em sublop} of $X$ if for any elements $a,b\in A$ we get $ab\in A$ and $a^{-1}b\in A$. By analogy we define subloops of loops.

Now we define the topological versions of the above notions. By a {\em topological magma} we will understand a topological space $X$ endowed with a continuous binary operation $\cdot:X\times X\to X$.
A topological magma $X$ is called
\begin{itemize}
\item a {\em topological left-loop} if $X$ has a right unit $e$ and the map $X\times X\to X\times X$, $(x,y)\mapsto (x,xy)$, is a homeomorphism;
\item a {\em topological lop} if $X$ has a unit $e$ and the map $X\times X\to X\times X$, $(x,y)\mapsto (x,xy)$, is a homeomorphism;
\item a {\em topological loop} is $X$ has a unit and the maps $X\times X\to X\times X,\;\;(x,y)\mapsto (x,xy)$,  and $X\times X\to X\times X$,  $(x,y)\mapsto (x,yx)$, are homeomorphisms.
\end{itemize}

\begin{remark} Topological loops are standard objects in (non-associative) topological algebra \cite{HS} whereas topological left-loops were recently introduced by Hofmann and Martin \cite{HM}. The notion of a topological lop seems to be new.
\end{remark}

It follows that for a topological left-loop $X$ with right unit $e$ the homeomorphism $H:X\times X\to X\times X$, $H:(x,y)\mapsto (x,y)$, witnesses that the topological space $X$ is rectifiable (since $H(x,e)=x$ and $H(\{x\}\times X)=\{x\}\times X$ for every $x\in X$).
We will show that the converse is also true: each rectifiable space is homeomorphic to a topological left-loop and even to a topological lop. Moreover, the rectifiability is equivalent to the continuous homogeneity defined as follows.

A topological space $X$ is defined to be {\em continuously homogeneous} if for any points $x,y\in X$ there is a homeomorphism $h_{x,y}:X\to X$ such that $h_{x,y}(x)=y$ and $h_{x,y}$ continuously depends on $x$ and $y$ in the sense that the map $H:X^3\to X^3$ defined by $H(x,y,z)=(x,y,h_{x,y}(z))$ is a homeomorphism.

More formally, a continuously homogeneous space can be defined as a space $X$ admitting a homeomorphism $H:X^3\to X^3$ such that $H(x,y,x)=y$ and $H(\{(x,y)\}\times X)=\{(x,y)\}\times X$ for any points $x,y\in X$. The equivalence $(1)\Leftrightarrow(4)$ in the following characterization was established by Uspenski\u\i\ \cite[Proposition 15]{Us}.

\begin{theorem}\label{t:rectilop} For a topological space $X$ the following conditions are equivalent:
\begin{enumerate}
\item $X$ is rectifiable;
\item $X$ is homeomorphic to a topological left-loop;
\item $X$ is homeomorphic to a topological lop;
\item $X$ is continuously homogeneous;
\end{enumerate}
\end{theorem}

\begin{proof} (1)$\Ra$(2)
Assume that $X$ is rectifiable. Then there exists a point $e\in X$ and a homeomorphism $H:X\times X\to X\times X$ such that $H(x,e)=(x,x)$ and $H(\{x\}\times X)=\{x\}\times X$ for all $x\in X$. Let $\pr_2:X\times X\to X$, $\pr_2:(x,y)\mapsto y$, denote the coordinate projection. Then the space $X$ endowed with the binary operation $p=\pr_2\circ H:X\times X\to X$ is a topological left-loop with the right unit $e$.

$(2)\Ra(3)$ Assume that $X$ is a topological left-loop and let $e$ be the right unit of $X$. Define the continuous binary operations $p:X\times X\to X$ and $q:X\times X\to X$ by the formulas $p(x,y)=x(e^{-1}y)$ and $q(x,y)=e(x^{-1}y)$ for any $x,y\in X$. We claim that $X$ endowed with the binary operation $p$ is a topological lop. Indeed, the map $H:X\times X\to X\times X$, $H:(x,y)\mapsto(x,p(x,y))$, is a homeomorphism with inverse $H^{-1}:(x,y)\mapsto (x,q(x,y))$. It remains to check that $e$ is a unit of the magma $(X,p)$. For this observe that $L_e(e)=ee=e$ and hence $e^{-1}e=L_e^{-1}(e)=e$. Then for every $x\in X$ we get $p(x,e)=x(e^{-1}e)=xe=x$ and $p(e,x)=e(e^{-1}x)=x$.

(3)$\Ra$(4) Assume that $X$ is a topological lop. The continuous homogeneity of $X$ is witnessed by the homeomorphism $H:X^3\to X^3$ defined by
$H(x,y,z)=y(x^{-1}z)$ for $(x,y,z)\in X\times X\times X$.

(4)$\Ra$(1) Assuming that $X$ is continuously homogeneous, fix any homeomorphism $H:X^3\to X^3$ such that $H(x,y,x)=y$ and $H(\{(x,y)\}\times X)=\{(x,y)\}\times X$ for all points $x,y\in X$. For any points $x,y\in X$ let $h_{x,y}:X\to X$ be the homeomorphism of $X$ such that $H(x,y,z)=(x,y,h_{x,y}(z))$ for all $z\in X$.

Fix any point $e\in X$ and observe that the homeomorphism $h:X\times X\to X\times X$ defined by $h(x,y)=h_{e,x}(y)$ for $(x,y)\in X\times X$ witnesses that the space $X$ is rectifiable.
\end{proof}

It is well-known that each open subgroup of a topological group is closed. The same fact holds for topological lops.

\begin{proposition}\label{p:sublop} Each open sublop $H$ of a topological lop $G$ is closed in $G$.
\end{proposition}

\begin{proof} Take any point $x\in\bar H$. Since $x^{-1}x=e\in H$, the (separate) continuity of the division operation yields a neighborhood $U_x\subset X$ of $x$ such that $U_x^{-1}x\subset H$. Choose any point $u\in U_x\cap H$ and observe that $u^{-1}x\in H$ and hence $x\in uH\subset HH=H$.
\end{proof}

We will use this simple fact to prove the following structural result. Let us recall that a topological space $X$ is {\em strongly paracompact}  if every open cover of  has a star-finite open refinement. By Smirnov Theorem 3.12 in \cite{Burke} each regular Lindel\"of space is strongly paracompact. Let us recall that for a topological space $X$ its {\em Lindel\"of number} $l(X)$ is defined as the smallest cardinal $\kappa$ such that each open cover of $X$ has a subcover of cardinality $\le \kappa$. A regular space $X$ is {\em cosmic} if it has countable network weight (equivalently, is a continuous image of a separable metric space).

\begin{theorem}\label{t:sum} If a topological lop $X$ contains an open cosmic sublop $H$, then $X$ admits a disjoint open cover refining the open cover $\{xH:x\in X\}$. Consequently, $X$ is a topological sum of cosmic subspaces and hence $X$ is submetrizable and strongly paracompact.
\end{theorem}

\begin{proof} By induction on cardinals $\alpha\le l(X)$ we will prove that each sublop $G\subset X$ with Lindel\"of number $l(G)\le\alpha$ containing the open sublop $H$ admits a disjoint open cover refining the open cover $\{xH:x\in G\}$.

To start the induction, assume that $G$ is a Lindel\"of sublop of $X$ containing the open sublop $H$. Since the space $G$ is Lindel\"of, the open cover $\{xH:x\in G\}$ has a countable subcover $\{x_nH\}_{n\in\w}$ for some sequence $\{x_n\}_{n\in\w}\subset G$. For every $n\in\w$ put $U_n=x_nH\setminus \bigcup_{k<n}x_kH$. By Proposition~\ref{p:sublop}, the open sublop $H$ is clopen in $X$, which implies that $U_n\subset x_nH$ is a clopen subset of $X$. So, $\U=\{U_n\}_{n\in\w}$ is a disjoint open cover of $X$ refining the cover $\{xH:x\in G\}$.

Now assume that for some uncountable cardinal $\kappa\le l(X)$ we have proved that any sublop $G\subset X$ with $l(G)<\kappa$ and $H\subset G$ admits a disjoint open cover refining the cover $\{xH:x\in G\}$.

Fix a sublop $G\subset X$ containing $H$ and having Lindel\"of number $l(G)=\kappa$. Choose a subset $D\subset G$ of cardinality $|D|\le l(G)\le\kappa$ such that $G=\bigcup_{x\in D}xH$.
Let $\{x_\alpha\}_{\alpha<\kappa}$ be an enumeration of the set $D$. For every ordinal $\alpha<\kappa$ denote by $H_\alpha$ the smallest sublop of $X$ containing the set $D_\alpha=H\cup\{x_\beta\}_{\beta<\alpha}$. Taking into account that each point $x\in H_\alpha$ has open neighborhood $xH\subset H_\alpha$ in $X$, we conclude that the sublop $H_\alpha$ is open in $X$. By Proposition~\ref{p:sublop}, the sublop $H_\alpha$ is closed. Observe that the space $D_\alpha$ has network weight $\nw(D_\alpha)<\kappa$. Since $H_\alpha$ is a continuous image of the space $\bigcup_{n\in\w}D_\alpha^n$, we conclude that $l(H_\alpha)\le\nw(H_\alpha)\le\max\{\aleph_0,\nw(D_\alpha)\}<\kappa$. Then by the inductive assumption, the sublop $H_\alpha$ admits a disjoint open cover $\U_\alpha$ refining the cover $\{xH:x\in H_\alpha\}$. Observe that the union $H_{<\alpha}=\bigcup_{\beta<\alpha}H_\beta$ is an open sublop of $X$ (being the union of the increasing chain of open sublops $(U_\beta)_{\beta<\alpha}$). By Proposition~\ref{p:sublop}, the sublop $H_{<\alpha}$ is closed in $X$. Then $\V_\alpha=\{U\setminus H_{<\alpha}:U\in\U_\alpha\}$ is a disjoint open cover of the space $H_\alpha\setminus H_{<\alpha}$ refining the cover $\U$. Unifying the covers $\V_\alpha$, $\alpha<\kappa$, we obtain the disjoint open cover $\V=\bigcup_{\alpha<\kappa}\V_\alpha$ of $G$ refining the cover $\{xH:x\in G\}$ of $G$.
\end{proof}

\section{Closed copies of the metric and Fr\'echet-Urysohn fans in rectifiable spaces}\label{s:fans}

 Banakh \cite{Ban98} proved that {\em a normal topological group $G$ is not sequential if $G$ contains closed copies of the fans $M_\w$ and $S_\w$.}  This result was extended in \cite[Lemma 4]{BZ} to arbitrary topological groups and  to all normal rectifiable spaces in \cite[Theorem 3.1]{LLL}.
In this section we will further generalize this Banakh's result to so-called $e$-normal topological lops, in particular, all topological lops of countable $\cs^*$-character.

\begin{definition}
 A topological lop $X$ is called {\em $e$-normal} if for any infinite closed discrete set $\{x_n\}_{n\in\w}$ in $X$ and any sequence $(y_m)_{m\in\w}$ in $X$ converging to $e$ there are increasing number sequences $(n_k)_{k\in\w}$ and $(m_k)_{k\in\w}$ such that the set $\{x_{n_k}y_{n_k}\}_{k\in\w}$ is sequentially closed in $X$.
\end{definition}

\begin{lemma}\label{l:enorm} A topological lop $X$ is $e$-normal if the topological space $X$ is normal or has countable $\cs^*$-character at $e$.
\end{lemma}

\begin{proof} To prove that $X$ is $e$-normal, fix an infinite closed discrete set $D=\{x_n\}_{n\in\w}$ in $X$ and a sequence $(y_m)_{m\in\w}$ in $X$ converging to $e$. We lose no generality by assuming that the points $x_n$, $n\in\w$, are pairwise distinct. If the set $\{m\in\w:y_m=e\}$ is infinite, then we can choose an increasing number sequence $(m_k)_{k\in\w}$ such that $\{m_k\}_{k\in\w}=\{m\in\w:y_m=e\}$ and conclude that the set $\{x_k,y_{m_k}\}_{k\in\w}=D$ is (sequentially) closed in $X$. So, we can assume that $y_m\ne e$ for all $m\in\w$.

1. If the space $X$ is normal, then by the Tietze-Urysohn Theorem~\cite[2.1.8]{En}, the (continuous) map $f:D\to \w\subset\IR$, $f:x_n\mapsto n$, extends to a continuous map $\bar f:X\to\IR$. Since for every $k\in\w$ the sequence $\big(f(x_ky_m)\big)_{m\in\w}$ converges to $f(x_k)=k$, we can choose a number $m_k\in\w$ such that $|f(x_ky_{m_k})-k|<\frac13$ and moreover $m_{k}>m_{k-1}$ if $k>0$. The continuity of the function $f$ guarantees that the sequence $(x_ky_{m_k})_{k\in\w}$ has no accumulation points in $X$ (since its image $\big(f(x_ky_{m_k})\big)_{k\in\w}$ has no accumulation points in $\IR$). Consequently, the set $\{x_ky_{m_k}:k\in\w\}$ is (sequentially) closed in $X$.

2. Now assume that the space $X$ has countable $\cs^*$-character at $e$ and fix a countable $\cs^*$-network $\mathcal N$ at $e$ such that each set $N\in\mathcal N$ contains the point $e$. Let $\{(A_l,B_l)\}_{l\in\w}$ be an enumeration of the countable family $\mathcal N\times\mathcal N$.
For numbers $k,m,l\in\w$ let $\overline{((x_ky_m)A_l)B_l}$ be the closure of the set $((x_ky_m)A_l)B_l$ in the topological lop $X$.

Since the (non-trivial) sequence $(x_0y_{m})_{m\in\w}$ converges to $x_0$, we can find a number $m_0\in\w$ such that $x_0y_{m_0}\notin D$.
For every $k\in\IN$  choose by induction a number $m_k>m_{k-1}$ such that the point $x_ky_{m_k}$ does not belong to the closed set
$$D_k=D\cup\bigcup\big\{\overline{((x_iy_{m_i})A_l)B_l}:i<k, \;l\le k,\;\;x_k\notin \overline{((x_iy_{m_i})A_l)B_l} \big\}.$$
We claim that the set $D'=\{x_ky_{m_k}\}_{k\in\w}$ is sequentially closed in $X$. Assuming the converse, we can find a convergent sequence $S\subset X$ such that $S\setminus\{\lim S\}\subset D'$ but $\lim S\notin D'$. Let $a=\lim S$ be the limit point of $S$. Since the set $D$ is closed and discrete in the regular space $X$, the point $a$ has a closed neighborhood $W_a\subset X$ such that $W_a\cap D\subset\{a\}$. Since $(ae)e=a\in W_a$ we can use the continuity of the multiplication in the topological lop $X$ and find two open sets $U_a\ni a$ and $U_e\ni e$ in $X$ such that $(U_aU_e)U_e\subset W_a$. Since the sequence $a^{-1}S$ converges to $e$, there is a set $B\in\mathcal N$ such that $B\subset U_e$ and the set $a^{-1}S\cap B$ is infinite. Replacing the sequence $S$ by its subsequence $S\cap aB$, we can assume that $S\subset aB$.
Observe that for each $z\in S\setminus\{a\}$ we get $z^{-1}a\ne e$ (in the opposite case, $a=ze=z$, which contradicts the choice of $z$). This implies that $S^{-1}a$ is an infinite sequence convergent to $e$. Then the $\cs^*$-network $\mathcal N$ contains a set $A\subset U_e$ that has infinite intersection with the sequence $S^{-1}a$. Since $S$ converges to $a\in U_a$, we can select a point $z=x_iy_{m_{i}}\in U_a\cap S\setminus\{a\}$ such that $z^{-1}a\in A\subset U_e$ and hence $a\in zA\subset U_aU_e$. Then $S\subset aB\subset (zA)B\subset (U_aU_e)U_e\subset W_e$ and hence $\overline{(zA)B}\subset \overline{W}_e=W_e\subset X\setminus (D\setminus\{a\})$. Find a number $l\in\w$ such that $(A,B)=(A_l,B_l)$.
 Since $S\setminus\{a\}\subset D'$ is infinite, we can find a number $k>\max\{i,l\}$ such that $x_k\ne a$ and  $x_ky_{m_k}\in S\subset (zA)B=((x_iy_{m_i})A_l)B_l$ but this contradicts the (inductive) choice of the number $m_k$ (as $x_k\notin \overline{(zA)B}$ and hence $x_ky_{m_k}\notin \overline{(zA)B}$~).
 This contradiction completes the proof of $e$-normality of $X$.
\end{proof}

The following theorem generalizes Theorem 4 of \cite{Ban98}, Lemma~4 of \cite{BZ} and Theorem 3.1 of \cite{LLL} to $e$-normal topological lops.

\begin{theorem}\label{t:fans} If an $e$-normal topological lop $X$ contains closed copies of the fans $M_\w$ and $S_\w$, then $X$ is not sequential.
\end{theorem}

\begin{proof}
 To derive a contradiction, assume that the space $X$ is sequential. By our assumptions, there are closed topological embeddings $\varphi:M_\w\to X$ and $\psi:S_\w\to X$.
 Since $X$ is topologically homogeneous, we can additionally assume that $\varphi(0,0)=e=\psi(0,0)$.
 For every $n,m\in\IN$ consider the points $x_{n,m}=\varphi\big(\frac1n,\frac1{nm}\big)$ and $y_{n,m}=\psi\big(\frac1m,\frac1{nm}\big)$. It follows that for every $n\in\w$, the sequence $(x_{n,m})_{m\in\w}$ has no accumulation point in $X$ whereas the sequence $(y_{n,m})_{m\in\w}$ converges to $e$.

Using the regularity of the space $X$, we can choose for every $n,m\in\IN$  a neighborhood $O(x_{n,m})\subset X\setminus\{e\}$ of the points $x_{n,m}$ such that the family $\big(O(x_{n,m})\big)_{n,m\in\w}$ is disjoint. Replacing each sequence $(y_{n,m})_{m\in\IN}$ by a suitable subsequence, we can assume that $x_{n,m}\cdot y_{n,k}\in O(x_{n,m})$ for every $n,m\in\IN$ and $k\ge m$.

Using the $e$-normality of the topological lop $X$, we can select for every $n\in\w$  an infinite subset $\Omega_n\subset\w$ and an increasing function $f_n:\Omega_n\to\w$ such that the set $A_n=\{x_{n,k}\cdot y_{n,f_n(k)}:k\in\Omega_n\}$ has finite intersection with any convergent sequence in $X$. Since $x_{n,k}\cdot y_{n,f_n(k)}\in O(x_{n,k})$ the points $x_{n,k}y_{n,f_n(k)}$, $n,k\in\IN$, are pairwise distinct and not equal to $e$.

Consider
the subset $A=\bigcup_{n\in\w}A_n$. Using
the continuity of the multiplication in $X$, one can show that $e\not\in A
$ is a cluster point of $A$ in $X$. Consequently, the set $A$ is
not closed and by the sequentiality of $X$, there is a sequence
$S\subset A$ converging to a point $a=\lim S\notin A$. Since every set
$A_n$ has finite intersection with $S$, we may replace $S$ by a
subsequence, and assume that $|S\cap A_n|\le 1$ for every $n\in
{\mathbb N}$. Consequently, $S$ can be written as
$S=\{x_{n_i,m_i}\cdot y_{n_i,m_i}:i\in\omega\}$ for some number
sequences $(m_i)$ and $(n_i)$ with $n_{i+1}>n_i$ for all $i$. It
follows (from the topological structure of the metric fan $M_\w$) that the sequence $(x_{n_i,m_i})_{i\in\omega}$ converges
to $e$ and consequently, the sequence
$T=\{y_{n_i,m_i}\}_{i\in\omega}=
\{x_{n_i,m_i}^{-1}x_{n_im_i}y_{n_i,m_i}\}_{i\in\w}$ converges to $e^{-1}(a)=e^{-1}(ea)=a$.
The definition of the topology of the Fr\'echet-Urysohn fan $S_\w$ guarantees that the sequence $T$ does not converge to $e$. Since $\psi(S_\w)\setminus\{e\}$ is a discrete space, the point $a$ does not belong to $\psi(S_\w)$, which means that $\psi(S_\w)$ is not closed in $X$. But this contradicts the choice of the closed embedding $\psi$.
\end{proof}

Lemma~\ref{l:enorm} and Theorem~\ref{t:fans} imply the following

\begin{corollary}\label{c:fans} If a topological lop $X$ has countable $\cs^*$-character and contains closed copies of the fans $M_\w$ and $S_\w$, then the space $X$ is not sequential.
\end{corollary}

\section{The Key Lemma}\label{s:key}

In this section we prove the key lemma that will be used in the proof of Theorem~\ref{main}. We will say that a topological space $X$ is {\em $S_\w$-vacuous} if $X$ contains no closed topological copy of the Fr\'echet-Urysohn fan $S_\w$.

\begin{keylemma}\label{l:key} Let $G$ be a topological lop and $F\subset G$ be a subset containing the unit $e$ of $G$. Put $F_1=F$ and $F_{n+1}=F_n^{-1}F_n$ for $n\in\IN$.
\begin{enumerate}
\item If $F$ is an $S_\w$-vacuous sequential space and each space $F_n$, $n\in\IN$, has a countable $\cs$-network at $e$, then $F$ has a countable $sb$-network at $e$.
\item If $F$ is sequential and each space $F_n$, $n\in\IN$, has countable $sb$-network at $e$, then $F$ is first countable at $e$.
\item If every space $F_n$, $n\in\IN$, is an $S_\w$-vacuous sequential space of countable $\cs^*$-character at $e$, then $F$ is first countable at $e$.
\end{enumerate}
\end{keylemma}

\begin{proof} 1. Assume that $F$ is an $S_\w$-vacuous sequential space and each space $F_n$, $n\in\IN$, has countable $\cs$-network at $e$. Then we can find a countable family $\mathcal{A}$ of subsets of $G$, which is a $\cs$-network at $e$ for each space $F_n$, $n\in\IN$. We can enlarge the family $\mathcal A$ to a countable family which contains all sets $F_n$, $n\in\IN$, and is closed under finite intersections, finite unions and taking products in the lop $G$.

We claim that the countable collection $\mathcal{A}|F=\{A\in \A:A\subset F\}$ is a $\mathrm{sb}$-network at $e$ in $F$.
Assuming the opposite, we could find an open neighborhood $
U\subset G $ of $ e $ such that no set $ A\in\mathcal{A}|F$
with $ A\subset U $ is a
sequential barrier at $ e $ in $ F $.

Consider the countable subfamily $\mathcal{A}'=\{A\in\mathcal{A}:A\cap F\subset U\}$ and let $\{A_n
:n\in\omega\}$ be an enumeration of $\A'$. For every $n\in\w$ put $ B_n =\bigcup_{k\le n}A_{k}$.

Let $m_{0}=0$ and $U_{0}\subset U$ be any closed neighborhood of $e$ in $G$. By
induction, for every $k\in\IN$ find a number $m_k>m_{k-1}$, a
closed neighborhood $U_k\subset U_{k-1}$ of $e$ in $G$, and a
sequence $(x_{k,i})_{i\in\IN}$ converging to $e$ such that for every $k\in\IN$ the
following conditions are satisfied:
\begin{itemize}
\item[(a)] $\{x_{k,i}:i\in\IN\}\subset U_{k-1}\cap F\setminus
B_{m_{k-1}}$;
\item[(b)] the set $E_k=\{x_{n,i}: n\le k,\; i\in\IN\}\setminus
B_{m_k}$ is finite;
\item[(c)] $U_k\cap(E_k\cup\{x_{i,j}:i,j\le k\})=\emptyset$ and
$U_kU_k\subset U_{k-1}$.
\end{itemize}

By induction we will prove that for every $k\in\IN$ and every $i\le k$ we get $U_i\cdots U_k\subset U_{i-1}$. For $i=k$ this follows from the inclusion $U_k=eU_k\subset U_kU_k\subset U_{k-1}$.
Assume that for some $i<k$ we have proved that $U_{i+1}\cdots U_k\subset U_{i}$. Then $U_{i}\cdots U_k=U_{i}(U_{i+1}\cdots U_k)\subset U_{i}U_{i}\subset U_{i-1}$ by the inductive assumption and the choice of $U_i$. Consequently,
$$U_1\cdots U_k\subset U_0\subset U.$$

It follows that the subspace $X=\{x_{k,i}:k,i\in\IN\}$ of $F$ is discrete and hence open in its closure $\bar X$ in $F$. Therefore the remainder $\bar X\setminus X$ is
closed in $F$.

\begin{claim}\label{cl2} The point $e$ is isolated in $\bar
X\setminus X$.
\end{claim}

\begin{proof} Assuming the converse and applying Lemma~\ref{l1}
we could find a sequence tree $T\subset \bar X$ such that $\min
T=\{e\}$, $\max T\subset X$, and $\mathrm{succ}(e)\subset \bar
X\setminus X$.

By induction, we will construct a (necessarily, finite) branch $(t_k)_{k\le n+1}$ of
the tree $T$, a sequence $\{C_k:k\le n\}$ of elements of the
family $\mathcal A$, and a sequence of points $(c_k)_{k=1}^{n+1}$ of $G$ such that for every $k\le n$ the following conditions are satisfied:
\begin{itemize}
\item[(d)] $c_k\in C_k\subset F_k\cap U_{k+1}$;
\item[(e)] $t_k=c_1\cdots c_k\in \suc(t_{k-1})$;
\item[(f)] $c_1\cdots c_kC_{k+1}$ contains almost all points of the sequence $\suc(t_k)$.
\end{itemize}

We start the inductive construction by letting $t_0=e$. Since the sequence $\suc(t_0)\subset F$ converges to $t_0$, we can find an element $C_0\in\A$ of the $\cs$-network such that $C_0\subset F\cap U_1$ and $\suc(t_0)\setminus C_0$ is finite. Assume that for some $k\ge 0$ the point $t_k=c_1\cdots c_k\in T\cap U$ has been chosen. If $t_k$ is a maximal point of the tree $T$, then we stop the construction.
So, we assume that $t_k$ is not maximal and hence $\suc(t_k)$ is an infinite sequence in $T\subset \bar X\subset F$ tending to $t_k$.

Since $c_k^{-1}\cdots c_1^{-1}(t_k)=e$, the sequence $c_k^{-1}\cdots c_1^{-1}(\suc(t_k))$ tends to $e$. We will show by induction  that for every $i\in\{1,\dots,k\}$ we get $c_i^{-1}\cdots c_1^{-1}(\suc(t_k))\subset F_{i+1}$.
For $i=1$ this follows from $c_1^{-1}(\suc(t_k))\subset C_1^{-1}F\subset F^{-1}F=F_2$. Assume that for some $i\le k$ we have proved that $c_{i-1}^{-1}\cdots c_1^{-1}(\suc(t_k))\subset F_i$. Then $$c_i^{-1}\cdots c_1^{-1}(\suc(t_k))=c_i^{-1}(c_{i-1}^{-1}\cdots c_1^{-1}(\suc(t_k)))\subset C_i^{-1}F_i\subset F_i^{-1}F_i=F_{i+1}.$$ Then for $i=k$ we get the desired inclusion
$c_k^{-1}\cdots c_1^{-1}(\suc(t_k))\subset F_{k+1}$.

Since $\A$ is a $\cs$-network at $e$ in $F_{k+1}\in\A$, there is a set $C_{k+1}\in\A$ which is contained in $F_{k+1}\cap U_{k+1}$ and contains all but finitely many points of the sequence $c_k^{-1}\cdots c_1^{-1}(\suc(t_k))$. Fix any point $t_{k+1}\in\suc(t_k)$ such that the point $c_{k+1}:=c_k^{-1}\cdots c_1^{-1}(t_{k+1})$ belongs to $C_{k+1}$. It follows that $t_{k+1}=c_1\cdots c_{k+1}$ and $\suc(t_{k})\setminus c_1\cdots c_kC_{k+1}$ is finite.
This completes the inductive step.

After completing the inductive construction, we obtain a branch $(t_k)_{k\le n+1}$ of the tree $T$ satisfying the conditions (d)--(f). Consider the (last by one) point $t_n$ in the branch $(t_k)_{k\le n+1}$. It follows that  $\suc(t_n)\subset\max T\subset X$. Since $\suc(e)\subset \bar X\setminus X$, the sequence $\suc(t_n)$ converges to the point $t_n\ne e$. By the inductive assumption, the set $\sigma:=\suc(t_n)\cap c_1\cdots c_nC_{n+1}$ contains all but finitely many points of the sequence $\suc(t_n)$. Consequently, $\sigma$ converges to $t_n$ as well.
On the other hand,
$\sigma\subset c_1\cdots c_nC_{n+1}\subset C_1C_2\cdots C_{n+1}\subset U_1U_2\cdots U_{n+1}\subset U_0\subset U$. It follows from the choice of $\A$ that $C_1\cdots C_{n+1}\in\A'$ and hence $(C_1\cdots C_{n+1})\cap F\subset B_{m_k}$ for some $k$. Consequently,
$\sigma\subset X\cap B_{m_k}$ and $\sigma\subset\{x_{j,i}:j\le
k,\;i\in\IN\}$ by the item (a) of the construction of $X$.
Since $e$ is a unique cluster point of the set $\{x_{j,i}:j\le
k,\; i\in\IN\}$, the sequence $\sigma$ cannot converge to
$t_n\ne e$, which is a contradiction completing the proof of Claim~\ref{cl2}.
\end{proof}

By Claim~\ref{cl2}, $e$ is an isolated point of $\bar X\setminus X$. Consequently, we can find a closed neighborhood $W$ of $e$ in $G$
such that the set $X_e=(\{e\}\cup X)\cap W$ is closed in $F$. Since the space $X$ is discrete,
the closed subspace $X_e$ has a unique non-isolated point $e$. Write the Fr\'echet-Urysohn fan
$S_\w$ as the union $S_\w=\{(0,0)\}\cup\bigcup_{n\in\IN}I_n$ of the convergent sequences $I_n=\{(\frac1i,\frac1{in}):i\in\IN\}$. Next, take any bijective map $h:S_\w\to X_e$ such that $h(0,0)=e$ and for every $n\in\IN$ the image $h(I_n)$ coincides with the sequence $J_n=X_e\cap\{x_{n,i}:i\in\IN\}$. By the choice of the topology on $S_\w$, the bijective map $h:S_\w\to X_e$ is continuous. Since the space $F$ is $S_\w$-vacuous, the map $h$ is not a homeomorphism, which allows us to find a closed set $D\subset S_\w$ whose image $h(D)$ is not closed in $X_e$.
By the sequentiality of $X_e$, there exists a sequence $S\subset \{s\}\cup h(D)$ converging to a point  $s\notin h(D)$. Since $h(D)$ is not closed in $X_e$ and $e$ is a unique non-isolated point of $X_e$, we get $e=s\notin h(D)$ and hence $(0,0)\notin D$.
Then the set $D$, being closed in $S_{\w}$ has finite intersection with each sequence $I_n=h^{-1}(J_n)$. Consequently the sequence $S\subset \{e\}\cup h(D)\subset \bigcup_{n\in\IN}J_n$ has finite intersection with each convergent sequence $J_n$, $n\in\IN$. Since $S$ is infinite it must meet infinitely many sequences $J_n$, $n\in\IN$.

On the other hand, the $\mathrm{cs}$-network
$\mathcal A'$ at $e$ contains a set $A_n$ containing almost all points of
the sequence $S$. Since $J_m\cap (J_k\cup A_n)=\emptyset$ for $m>
k\ge n$, the sequence $S$ cannot meet infinitely many sequences
$J_m$. This is a contradiction showing that
$\mathcal{A}|F$ is a $\mathrm{sb}$-network at $e$ in $F$.

2. Assume that the space $F$ is sequential and each space $F_n$, $n\in\IN$, has countable $sb$-network at $e$.

\begin{claim}\label{cl5.3} For each sequential barrier $U\subset F^{-1}F$ at $e$ the intersection $U\cap F$ is a neighborhood of $e$ in $F$.
\end{claim}

\begin{proof} Assuming that for some sequential barrier $U\subset F^{-1}F$ the set $U\cap F$ is not a neighborhood of $e$ at $F$, we can apply Lemma \ref{l1} and find a sequence tree $T\subset F $ such that $\min T=\{e\}$ and $\max T\subset F\setminus U$. To get a
contradiction it suffices to construct an infinite branch of $ T $.

Let $S_1=U\cap F$. By induction  we will construct for every $n>1$ a sequential barrier $S_n$ at $e$ in the space $F_n$ such that $S_nS_n\cap F_{n-1}\subset S_{n-1}$. Assume that for some $n> 0$ a sequential barrier $S_{n-1}\subset F_{n-1}$ has been constructed.

By our assumption, the space $F_n$ has countable $\sbr$-character at $e$ and hence admits a decreasing $\mathrm{sb}$-network $\{B_{m}\}_{m\in\w}$ at $e$. We claim that $B_mB_m\cap F_{n-1}\subset S_{n-1}$ for some $m\in\w$. In the opposite case for every $ m\in\w$ we can find points $ x_m
,y_m \in B_m$ such that $x_m y_m\in F_{n-1}\setminus S_{n-1}$. Taking
into account that $\lim_{m\rightarrow \infty } x_m =
\lim_{m\rightarrow\infty } y_m =e $, we get
$\lim_{m\rightarrow\infty}x_m y_m =e $. Since $S_{n-1}$ is a
sequential barrier at $e$, there is a number $m$ with $x_my_m\in
S_{n-1}$, which contradicts the choice of the points $x_m,y_m$.
This contradiction shows that $B_mB_m\cap F_{n-1}\subset S_{n-1}$ for some $m$ and we can put $S_{n}=B_m$.

Now we are ready to construct an infinite branch $(t_i)_{i\in\w}$ of the tree $T$. Put $t_0=e\in T$. By induction we will construct sequences of points $(t_n)_{n\in\w}$ and $(s_n)_{n\in\IN}$ in $G$ such that the following conditions are satisfied for every $n\in\IN$:
\begin{itemize}
\item[(g)] $t_n\in\suc(t_{n-1})$;
\item[(h)] $s_n\in S_{n+1}\cap F_n$;
\item[(i)] $t_n=s_1\cdots s_n\in U\cap T$.
\end{itemize}
Assume that for some $n>0$ the point $t_n=s_1\cdots s_n\in U\cap T$ has been constructed.
Since $t_n\in U\cap T$ and $\max(T)\subset F\setminus U$, the point $t_n$ is not maximal in $T$ and hence the convergent sequence $\suc(t_n)\subset T$ is well-defined. Since $s^{-1}_n\cdots s_1^{-1}(t_n)=e$, the sequence  $s^{-1}_n\cdots s_1^{-1}(\suc(t_n))$ converges to $e$. By induction it can be shown that this sequence is contained in the set $F_{n+1}$. Since $S_{n+2}\cap F_{n+1}$ is a sequential barrier at $e$ in $F_{n+1}$, there is a point $t_{n+1}\in\suc(t_n)$ such that the point $s_{n+1}:=s^{-1}_n\cdots s_1^{-1}(t_{n+1})$ belongs to the barrier $S_{n+2}$. Multiplying this equality by $s_{n},\dots,s_1$ from the left, we get the equality $s_{1}\cdots s_{n+1}=t_{n+1}$. It remains to show that $t_{n+1}\in U$.

By induction on $k\le n+1$ we will prove that $s_k\cdots s_{n+1}\subset F_k$. For $k=1$ this is true as $t_{n+1}=s_1\cdots s_{n+1}\in F=F_1$. Assume that the inclusion  $s_k\cdots s_{n+1}\in F_k$ has been proved for some $k<n+1$. Then $s_{k+1}\cdots s_{n+1}=s_k^{-1}(s_k\cdots s_{n+1})\in F_k^{-1}F_k=F_{k+1}$.

Next, we will show by induction on $k\in\{n+1,\dots,1\}$  that $s_k\cdots s_{n+1}\in S_{k}$. For $k=n+1$ this follows by the choice of the point $s_{n+1}\in S_{n+2}\cap F_{n+1}\subset S_{n+1}$. Assume that for some $k\le n+1$ we have proved that $s_{k+1}\cdots s_{n+1}\in S_{k+1}$. The choice of $s_k\in S_{k+1}$ and the inductive assumption guarantee that $s_k(s_{k+1}\cdots s_{n+1})\in F_k\cap (S_{k+1}S_{k+1})\subset S_k$. So, $t_{n+1}=s_1\cdots s_{n+1}\in S_1\subset U$.

Therefore we have constructed an infinite branch
$\{t_i:i\in\omega\}$ of the sequence tree $T$ which is impossible.
\end{proof}

Fix any countable $sb$-network $\mathcal B$ at $e$ in the space $F^{-1}F$. Claim~\ref{cl5.3} implies that $\{B\cap F:B\in\mathcal B\}$ is a countable neighborhood base at $e$ in the space $F$.

3. Assume that for every $n\in\IN$ the space $F_n$ is an $S_\w$-vacuous sequential space with countable $\cs^*$-network at $e$. For every $n\in\IN$ put $E=F_n$, $E_1=E$ and $E_{m+1}=E_m^{-1}E_m$ for $m\ge 1$.
 It can be shown by induction that $E_m=F_{n+m-1}$ for every $m\in\IN$. So, the spaces $E_m$, $m\in\IN$, have countable $\cs^*$-network at $e$. By Lemma~\ref{l:cs=cs*}, these spaces has countable $\cs$-network at $e$.
Applying Lemma~\ref{l:key}(1) to the $S_\w$-vacuous space $E=F_n$ and the sequence $(E_m)_{m\in\IN}$, we conclude that $F_n=E$ has countable $sb$-network at $e$. By Lemma~\ref{l:key}(2), the space $F$ is first countable at $e$.
\end{proof}

\section{Metrizability criteria for subsets of rectifiable spaces}\label{s:metr}

In this section we will derive some corollaries from  Key Lemma~\ref{l:key}, which can be interesting on their own. The first of them generalizes several metrizability criteria proved in \cite[\S4]{Lin13} and \cite[\S4]{LS}.

\begin{theorem}\label{t:mc} A rectifiable space $X$ is metrizable if and only if it is sequential, has countable $\cs^*$-character, and contains no closed copy of the Fr\'echet-Urysohn fan $S_\w$.
\end{theorem}

\begin{proof} The ``only if'' part is trivial. To prove that ``if'' part, assume that $X$ is a rectifiable sequential $S_\w$-vacuous space of countable $\cs^*$-character. By Theorem~\ref{t:rectilop}, we can assume that $X$ is a topological lop.
By Lemma~\ref{l:key}(3), the topological lop $X$ is first countable at $e$. By \cite{Gul}, the rectifiable space $X$, being first countable, is metrizable.
\end{proof}

Another corollary of Lemma~\ref{l:key} concerns metrizability of countably or sequentially compact subsets in rectifiable spaces of countable $\cs^*$-character.

\begin{theorem}\label{t:ccm} Let $X$ be a sequential rectifiable space of countable $\cs^*$-character. Then each countably compact subset of $X$ is metrizable.
\end{theorem}

\begin{proof} By Theorem~\ref{t:rectilop}, we can assume that $X$ is a topological lop. Fix any countably compact subset $F\subset X$. Since the space $X$ is sequential, $F$ is sequentially compact. Put $F_1=F$ and $F_{n+1}=F_n^{-1}F_n\subset X$ for $n\ge 1$. Since the sequential compactness is preserved under finite products and continuous images, the spaces $F_n$ are sequentially compact for all $n\in\IN$. For every $n\in\IN$ the sequential compactness of $F_n$ and the sequentiality of $X$ imply the sequentiality of the space $F_n$. Then  Lemma~\ref{l:key} implies that all spaces $F_n$ are first countable at $e$. In particular, the space $F_2=F^{-1}F$ is first countable at $e$. The continuity of the division $q:F\times F\to F^{-1}F$, $q:(x,y)\mapsto x^{-1}y$, implies that the square $F\times F$ has $G_\delta$-diagonal $\Delta_F=\{(x,t)\in F\times F:x=y\}=q^{-1}(e)$. By Chaber's Theorem \cite[2.14]{Gru}, the countably compact space $F$  is metrizable.
\end{proof}

The sequentiality of the rectifiable space $X$ in the preceding corollary can be replaced by the sequentiality and compactness of the space $F$.

\begin{theorem}\label{t:csm}  Let $X$ be a rectifiable space of countable $\cs^*$-character. Then each compact sequential subspace of $X$ is metrizable.
\end{theorem}

\begin{proof} Let $F\subset X$ be a compact sequential space. Put $F=F_1$ and $F_{n+1}=F_n^{-1}F_n$ for $n\ge 1$. By induction we will prove that every space $F_n$ is compact and sequential.
For $n=1$ this follows from the sequentiality and compactness of $F$.
Assume that for some $n\in\IN$ we have proved that the space $F_n$ is compact and sequential. Then each closed subset of $F_n$ is sequentially compact and hence the square $F_n\times F_n$ is a compact sequential space (see \cite[3.10.I(b)]{En}). Taking into account that the space $F_n\times F_n$ is compact and the map $q:F_n\times F_n\to F_n^{-1}F_n$, $q:(x,y)\mapsto x^{-1}y$ is continuous, we conclude that the image $F_{n+1}=F_n^{-1}F_n=q(F_n\times F_n)$ is compact and the map $q$ is closed and thus quotient. Since the sequentiality is preserved by quotient maps (\cite[2.4.G]{En}), the space $F_{n+1}$ is sequential. Therefore all spaces $F_n$, $n\in\w$, are sequential and, being compact, are sequentially compact and $S_\w$-vacuous. Now it is possible to apply Key Lemma~\ref{l:key} to conclude that the space $F_2=F^{-1}F$ is first countable at $e$, which implies that the compact space $F$ has $G_\delta$-diagonal and hence is metrizable by \cite[2.13]{Gru}.
\end{proof}

\section{The structure of non-metrizable topological lops of countable $\cs^*$-character}\label{s:nonmetr}

In this section we establish an algebraic analogue of Theorem~\ref{main}. More precisely, to deduce Theorem~\ref{main} we need to combine Theorems~\ref{t:rectilop} and \ref{t:sum} with the following theorem. For a topological lop $X$ an {\em $\sk_\w$-sublop} is a sublop of $X$ which is an $\sk_\w$-space.

\begin{theorem}\label{t:nonmetr} Each non-metrizable sequential  topological lop $X$ of countable $\cs^*$-character contains a clopen $\sk_\w$-sublop $U$.
\end{theorem}

\begin{proof} Since $X$ is not metrizable, we can apply Theorem \ref{t:mc} and conclude that the sequential space $X$ contains a closed topological copy of the Fr\'echet-Urysohn fan $S_\w$. Then by  Corollary~\ref{c:fans}, the space $X$ contains no closed
copy of the metric fan $M_\w$.


Fix a countable $\mathrm{cs}^*$-network $\mathcal{N} $ at $ e $, closed under finite unions, finite
intersections, and consisting of closed subspaces of $ X $. By (the proof of) Lemma~\ref{l:cs=cs*}, the family $\mathcal N$ is a $\cs$-network at $e$.
Consider the collection $\K\subset\mathcal N$ of all
countably compact subspaces $N\in\mathcal N$.

\begin{claim}\label{cl7.2} The family $\K$ is a $\mathrm{cs}$-network at $e$.
\end{claim}

\begin{proof} Given any neighborhood $U\subset X$ of $e$ and sequence $\{x_n\}_{n\in\w}\subset X$ converging to $e$, we should find a
countably compact set $K\in\mathcal N$ with $K\subset U$,
containing all but finitely many points the sequence. Let
$\mathcal{A}=\{A_k:k\in\omega\}$ be the collection of all elements
$N\subset U$ of $\mathcal N$ containing almost all points $x_n$.
Now it suffices to find a number $ n\in\omega $ such that the
intersection $K=\bigcap_{k\le n}A_k$ is countably compact. Suppose
to the contrary, that for every $ n\in\omega$ the set
$\bigcap_{k\le n}A_k$ is not countably compact. Then there exists
a countable closed discrete subspace $ D_{0}\subset A_{0}$ with
$D_0\not\ni e$. Fix a neighborhood $ W_{0}$ of $e$ such that $
W_{0}\cap D_{0}=\emptyset $. Since $\mathcal N$ is a
$\mathrm{cs}$-network at $e$, there exists $ k_{1}\in\omega $ such
that $ A_{k_{1}}\subset W_{0}$.

It follows from our hypothesis that there is a countable closed
discrete subspace $D_{1}\subset\bigcap_{k\le k_1}A_k$ with $D_1\not\ni
e$. Proceeding in this fashion we construct by induction an
increasing number sequence $ (k_{n})_{n\in\omega}\subset\omega $,
a sequence $(D_{n})_{n\in\omega}$ of countable closed discrete
subspaces of $G$, and a sequence $ (W_{n})_{n\in\omega}$ of open
neighborhoods of $ e $ such that $D_{n}\subset\bigcap_{k\le
k_n}A_{k}$, $W_{n}\cap D_{n}=\emptyset$, and $A_{k_{n+1}}\subset
W_{n}$ for all $ n\in\omega $.

It follows from the above construction that $M=
\{e\}\cup\bigcup_{n\in\omega}D_n$ is a closed topological copy of the metric fan $M_\w$ which is impossible.
\end{proof}

So, the family $\K$ is a $\cs$-network at $e$.
By Theorem~\ref{t:ccm}, each countably compact subspace of $X$ is metrizable (and hence compact).
Therefore, $\K$ is a $\cs$-network at $e$ consisting of compact metrizable subsets. Replacing $\K$ by a larger family, we can assume that $\K$ is closed under finite unions, finite intersections, and for any sets $A,B\in \K$ the (compact metrizable) sets $AB$ and $A^{-1}B$ belong to $\K$.
Then the union $U=\bigcup\K$ is a $\sigma$-compact sublop of the lop $X$.

Let us show that for each point $u\in U$ the set $U$ is a sequential barrier at $u$. Take any sequence $(x_n)_{n\in\w}$ in $X$, converging to the point $u$ in $X$. Since the left shift by $u$ is a homeomorphism of $X$ sending $e$ to $u$, the sequence $(u^{-1}x_n)_{n\in\w}$ converges to the unit $e$ of $X$. The family $\K$, being a $\cs$-network at $e$, contains a set $K\in\K$ containing all but finitely many points of the sequence $(u^{-1}x_n)_{n\in\w}$. Then the set $uK\subset U$ contains all but finitely many points of the sequence $(x_n)_{n\in\w}$, which means that $U$ is a sequential barrier at $u$. Since $X$ is sequential, the set $U$ is open in $X$. By Proposition~\ref{p:sublop}, the sublop $U$ is closed in $X$.

We claim that $U$ is an $sk_\w$-space (whose topology is generated by the cover $\mathcal{K} $). Indeed, consider any subset $F\subset U$
such that  for each compact set $ K\in\mathcal{K}$ the intersection $F\cap K $ is closed in $ K $. We need to check that $F$ is relatively closed in $U$. Assuming that $F$ is not closed in $U$, we conclude that $F\cup(X\setminus U)$ is not closed in $X$ and by the
sequentiality of $X$, we could find a sequence $ (x_n)_{n\in\omega}\subset F$ convergent to a point $ x\in U\setminus F $. It follows that there are elements $ K_1 ,K_2
\in\mathcal{K} $ such that $ x\in K_1 $ and $K_2$ contains almost
all members of the sequence $\{x^{-1}x_n\}_{n\in\w}$. Then the product $
K=K_1 K_2\in\K $ contains all points $x_n$, $n\ge n_0$, for some number $n_0$. Since the set $ F\cap K $ is closed, the limit point $x$ of the sequence $\{x_n\}_{n\ge n_0}\subset F\cap K$ belongs to $F\cap K\subset F$, which is a desirable contradiction showing that $U$ is an  $sk_\omega$-space.
\end{proof}

\end{document}